\documentclass{article}
\title{Some examples of extremal triples of commuting contractions}
\author{Edward J. Timko \\ Indiana University Department of Mathematics}

\usepackage{enumerate} 
\usepackage{amsmath}  %
\usepackage{amsfonts} 
\usepackage{amssymb}  %
\usepackage{amsthm}   %

 \newcommand{\CC}[0]{\mathbb{C}}  
 \newcommand{\NN}[0]{\mathbb{N}}  

\newcommand{\inner}[2]{ \langle #1 , #2 \rangle} 
\DeclareMathOperator{\Ker}{Ker}
\DeclareMathOperator{\Ran}{Ran}
\DeclareMathOperator{\ran}{ran}
\DeclareMathOperator{\diag}{diag}
\DeclareMathOperator{\Span}{Span}

\newcommand{\pd}{\partial}
\newcommand{\sbse}{\subseteq}

\newcommand{\frk}[1]{\mathfrak{#1}} 
\newcommand{\cc}[1]{\overline{#1}} 
\newcommand{\qtextq}[1]{\quad\text{#1}\quad} 
\newcommand{\pvect}[1]{\begin{pmatrix} #1 \end{pmatrix}} 
\newcommand{\of}{\circ}
\newcommand{\mc}[1]{\mathcal{#1}} 

\theoremstyle{plain}
\newtheorem{theorem}{Theorem}[section]
\newtheorem{lemma}[theorem]{Lemma}
\newtheorem{proposition}[theorem]{Proposition}

\theoremstyle{remark}
\newtheorem{remark}[theorem]{Remark}

\makeatletter
\@addtoreset{theorem}{section}
\makeatother

\numberwithin{equation}{section} 

\begin{document}

\maketitle

\begin{abstract}
    The collection $\frk{C}_3$ of all triples of commuting contractions forms a family in the sense of Agler, and so has an ``optimal'' model $\pd\frk{C}_3$ generated by its extremal elements. A given $T\in\frk{C}_3$ is extremal if every $X\in\frk{C}_3$ extending $T$ is an extension by direct sum. We show that many of the known examples of triples in $\frk{C}_3$ that fail to have coisometric extensions are in fact extremal.
\end{abstract}

\section{Introduction}

Given $n\in\NN$, we denote by $\frk{C}_n$ the class of all $n$-tuples of commuting contractions. Observe that $\frk{C}_n$ is a \textit{family} in the sense of Alger \cite{Agler1}, which is to say that:
      \begin{enumerate}[{(}i)]
          \item $\frk{C}_n$ is closed with respect to direct sums. That is, given $A^{(j)}\in\frk{C}_n$ for every $j\in J$, we have $\left(\bigoplus_{j\in J} A^{(j)}_i\right)_{i=1}^n \in\frk{C}_n$;
          \item Given $A\in\frk{C}_n$ and a unital $*$-representation $\pi$ of the unital C$^*$-algebra generated by $A_1,\dots, A_n$, then $(\pi(A_i))_{i=1}^n \in \frk{C}_n$; and
          \item $\frk{C}_n$ is hereditary. That is, if $A\in\frk{C}_n$ and if $\mc{M}$ is an invariant subspace of $\mc{H}$ for $A_1,\dots ,A_n$, then $(A_i|\mc{M})_{i=1}^n\in\frk{C}_n$.
      \end{enumerate}
An element $T\in\frk{C}_n$ is said to be \textit{extremal} if whenever $S\in\frk{C}_n$ is an extension of $T$, then $S$ is an extension by direct sum. That is, if $\mc{N}$ is invariant for $S$ so $T=S|\mc{N}$, then $\mc{N}$ is a reducing subspace for $S$. We say that $S$ is a \textit{trivial} extension of $T$ if $S$ is an extension of $T$ by direct sum.

Let $\mc{B}\sbse\frk{C}_n$. We say that $\mc{B}$ is a \textit{model} for $\frk{C}_n$ if
 \begin{enumerate}[{(}i)]
     \item $\mc{B}$ is closed with respect to direct sums and unitary $*$-representations; and
     \item given $T\in\frk{C}_n$ acting a Hilbert space $\mc{H}$, there exists $S\in\mc{B}$ having $\mc{H}$ as an invariant subspace so that $S|\mc{H}=T$.
 \end{enumerate}
Lastly, the \textit{boundary} of $\frk{C}_n$, denoted by $\pd\frk{C}_n$, is the smallest model for $\frk{C}_n$. It follows as a consequence of Proposition 5.9 and 5.10 in \cite{Agler1} that this family always exists and is generated by the extremal elements of $\frk{C}_n$.

\vspace{10pt}

In the case that $n=1,2$, the boundary $\pd\frk{C}_n$ consists of all tuples of commuting coisometries, a consequence of the work of Sz.-Nagy for $n=1$ and And\^{o} for $n=2$ \cite{SzNagy1}. For $n>2$ this characterization is no longer valid.  It may well be the case that no concrete description of the extremal elements of $\frk{C}_3$ or $\pd\frk{C}_3$ is possible. We show that many of the known examples of triples in $\frk{C}_3$ that fail to have coisometric extensions are in fact extremal.

\vspace{10pt}

Agler's theory has seen some application. Dritschel and McCullough show in \cite{DritMcCu1} that if $\mc{F}$ denotes the family of contractive hyponormal operators, then $\pd \mc{F}=\mc{F}$. In the same article, sufficient conditions are given for an $n$-hyponormal operator to be extremal. In an article by Curto and Lee \cite{CurtoLee}, it is shown that a weakly subnormal operator satisfying the conditions of \cite{DritMcCu1} must be normal and so extremal for the collection of all weakly subnormal operators. Dritschel, McCullough, and Woerdeman \cite{DritMcCuWo} give a collection of equivalent conditions for a $\rho$-contraction (for $\rho\leq 2$) to be extremal, ultimately showing for $\rho\in(0,1)\cup(1,2]$ that $\mc{C}_\rho=\pd\mc{C}_\rho$, with $\mc{C}_\rho$ denoting the class of $\rho$-contractions. In another article by Dritschel and McCullough \cite{DritMcCu2} it is shown that a family, in ``Agler's sense'', of representations of either an operator algebra or an operator space has boundary representations, as related to the non-commutative Shilov boundary. Finally, in \cite{RichSund} Richter and Sundberg apply Agler's theory to the study of row contractions and spherical contractions.

\vspace{10pt}

Here is an outline of the material found in this paper. In Section 2 we make some observations that apply to any $n$-tuple of commuting contractions. While these results are only applied in Section 5, they are general enough to merit separate exposition. In Section 3 we study an $n$-tuple of Parrot \cite{Par1}, finding the $n$-tuple to be extremal if and only if a certain subspace is trivial. In Section 4 we prove a triple of Crabb and Davie \cite[p.~23]{Pisier} is extremal in $\frk{C}_3$. In Section 5 we examine a triple due to Varopoulos \cite[p.~86]{Pisier} and show that this triple is extremal only for a relatively narrow range of parameters.

\vspace{10pt}

We comment on another triple that has appeared in the literature. In \cite{LotSte} Lotto and Steger find a triple of commuting, diagonalizable contractions that fail to obey the von Neumann inequality. This triple does not appear to produce extremal elements so its examination has been omitted from this paper.

\vspace{10pt}
Before closing this section, I would like to thank Hari Bercovici for his guidance and the helpful criticism he provided in the preparation of this paper, and to thank the referee for his helpful suggestions.

\section{Some General Remarks}

Lacking a complete description of the boundary elements, we develop some tools to tell us when we may determine when certain elements are not extremal. For the first lemma, we use the notation $\Ran T:=\bigvee_i \ran T_i$ and $\Ker T:=\bigcap_i \ker T_i$.

\begin{lemma}\label{GR:L1}
 Let $T\in\frk{C}_n$ operating on a Hilbert space $\mc{H}$. If $(\Ran T)^\bot\cap\Ker T\neq \{0\}$, then $T$ is not extremal.
\end{lemma}
\begin{proof}
 Let $\mc{E}:=(\Ran T)^\bot\cap\Ker T$ and $V:\mc{E}\to\mc{H}$ the inclusion map. Define $X$ on $\mc{H}\oplus\mc{E}$ by
 \[ X_i:=\pvect{T_i & V\\ 0 & 0} \quad i=0,1,\dots,n. \]
 As $T_iV=0$ for each $i$, the $X_i$ commute. Since $VV^*$ is orthogonal to the range of each $T_i$, it follows that $T_iT_i^*+VV^*\leq 1$, and therefore each $X_i$ is a contraction. Since $\mc{E}\neq 0$, $X_i$ is a non-trivial extension.
\end{proof}

\begin{lemma}\label{GR:L2}
 If $T\in\frk{C}_n$ satisfies $\min_i\|T_i\|<1$, then $T$ is not extremal.
\end{lemma}
\begin{proof}
 Consider the extension
 \[ X_i=\pvect{T_i & \delta_i T_i \\ 0 & \eta_i T_i} \quad i=0,1,\dots,n \]
 where $\delta_i,\eta_i\in [0,1]$ are to be determined. We want $X$ to be in $\frk{C}_n$. Note that
 \[ X_iX_j=\pvect{T_iT_j & (\delta_j+\delta_i\eta_j)T_iT_j \\ 0 & \eta_i\eta_j T_iT_j}. \]
 and therefore $X_iX_j=X_jX_i$ when either $\delta_j+\delta_i\eta_j=\delta_i+\delta_j\eta_i$ or $T_iT_j=0$ for each $i,j$. It suffices to set $\eta_i=1-\delta_i$ for each $i$.
 
 Observe now that
 \[ X_i^*X_i=\pvect{T_i^*T_i & \delta_i T_i^*T_i\\ \delta_i T_i^*T_i & (\delta_i^2+\eta_i^2)T_i^*T_i} \]
 Setting $\beta_i:=1+\delta_i^2+\eta_i^2$, we easily see
 \begin{equation}\label{GR:L2:Eq1}
     \|X_i\|^2\leq \frac{1}{2}\left[\beta_i+\sqrt{\beta_i^2-4\eta_i^2}\right]\|T_i\|^2.
 \end{equation}
 To conclude the proof, we show that the $\delta_i$ can be chosen so that the right-hand side of \eqref{GR:L2:Eq1} is at most 1 for each $i$. This is equivalent to insisting
 \[ \beta_i-\|T_i\|^2\eta_i^2\leq \frac{1}{\|T_i\|^2} \]
 or equivalently
 \begin{equation*}
 \delta_i^2+(1-\|T_i\|^2)(\eta_i^2-1)\leq \|T_i\|^2+\|T_i\|^{-2}-2=\left\{\frac{1-\|T_i\|^2}{\|T_i\|}\right\}^2. \end{equation*}
 Since $(1-\|T_i\|^2)(\eta_i^2-1)\leq 0$, fix
  \[\delta_i = \min\left\{1,\frac{1-\|T_i\|^2}{\|T_i\|}\right\}. \]
 As $\|T_i\|<1$ for some $i$ we have $\delta_i>0$.
\end{proof}

\section{Parrot's Example}

Parrott provided the first example of a triple of commuting contractions which has no commuting coisometric extension (\cite{Par1}; see also \cite[p.~23]{SzNagy1}). Let $U_1,\dots,U_n$ be an arbitrary $n$-tuple of unitaries operators on a Hilbert space $\mc{H}$, and define
\begin{equation}\label{PE:Eq1}
 T_i:=\pvect{0 & 0 \\ U_i & 0} \quad i=1,2,\dots,n
\end{equation}
acting on $\mc{H}\oplus\mc{H}$. It is easily checked that the $T_i$ are commuting partial isometries, and so $T\in\frk{C}_n$. When the $U_i$ do not commute ``enough'', then $T$ has no extension to an $n$-tuple of commuting coisometries. In particular, if for some $i\neq j$ the commutator $[U_n^{-1}U_i,U_n^{-1}U_j]$ does not vanish, then $T$ has no coisometric extension (here and elsewhere $[X,Y]=XY-YX$). We refer the reader to \cite[p.~23]{SzNagy1} for details in the case $n=3$. A similar criterion determines when $T$ is extremal.

\begin{proposition}
 Let $T$ denote the Parrott $n$-tuple defined by unitaries $U_1,\dots,U_n$ acting on a Hilbert space $\mc{H}$. Then $T$ is extremal if and only if
 \[ \bigcap_{i,j=1}^n \ker[U_n^{-1}U_i,U_n^{-1}U_j]=\{0\}. \]
\end{proposition}
\begin{proof}
 An extension $X$ of $T$ takes the form
 \[ X_i=\pvect{0 & 0 & A_i \\ U_i & 0 & B_i \\ 0 & 0 & C_i}, \quad i=1,2,\dots,n. \]
 As $X_i$ is contractive, $\pvect{U_i^*\\B_i^*}$ is also a contraction, and so $B_i= 0$ for each $i$. Therefore
 \[ X_i^*X_i=\pvect{1 & 0 & 0 \\ 0 & 0 & 0 \\ 0 & 0 & A_i^*A_i+C^*_iC_i} \]
 and so $X_i$ is a contraction if and only if $A_i^*A_i+C_i^*C_i\leq 1$. Since
 \[ X_iX_j=\pvect{0 & 0 & A_iC_j \\ U_iU_j & 0 & U_iA_j \\ 0 & 0 & C_iC_j} \quad i,j=1,2,\dots n, \]
 commutivity requires
 \[ A_iC_j=A_jC_i, \quad U_iA_j=U_jA_i, \quad [C_i,C_j]=0 \]
 for all $i,j$. Using the notation $W_j=U_n^{-1}U_j$ for $j=1,2,\dots,n$, the second of these implies $A_j=W_jA_n$ for each $j$ and therefore
 \[  [W_i,W_j]A_n=0, \quad i,j=1,2,\dots,n. \]
 Define $\mc{K}=\bigcap_{i,j}\ker[W_i,W_j]$. Thus $\ran A_n\sbse \mc{K}$. If $\mc{K}=0$, then $A_n=0$ hence $A_i=0$ for each $i$ and therefore every extension $X$ is by direct sum; $T$ is extremal.
 
 Conversely, suppose that $\mc{K}\neq \{0\}$. Let $A_n$ denote the canonical embedding of $\mc{K}$ into $\mc{H}$. Then
 \[ X_i=\pvect{0 & 0 & W_iA_n \\ U_i & 0 & 0 \\ 0 & 0 & 0}, \quad i=1,2,\dots,n \]
 defines a non-trivial extension of $T$ in $\frk{C}_n$.
\end{proof}

\begin{remark}
    The conditions given by \eqref{PE:Eq1} seem to unfairly favor $U_n$. The favoritism is in fact superficial. While this can be seen as a corollary of the preceding proposition, one may also directly show that $\bigcap_{i,j=1}^n\ker[U_n^{-1}U_i,U_n^{-1}U_j]=\{0\}$ if and only if $\bigcap_{i,j=1}^n\ker[U_k^{-1}U_i,U_k^{-1}U_j]=\{0\}$ for some $k=1,\dots,n$.
\end{remark}

\begin{remark}
    It should be noted that the condition in \eqref{PE:Eq1} can indeed be satisfied by some tuple of operators. Consider $g_1,g_2$, generators of the free group on two elements, acting by translation on $\ell^2(\mathbb{F}_2)$. Now consider the triple of unitaries $(g_1,g_2,1)$. The intersection in \eqref{PE:Eq1} reduces to $\ker[g_1,g_2]=\{0\}$.
\end{remark}

\section{The Crabb-Davie Example}

While Parrot's example has no coisometric extension in $\frk{C}_n$, it nevertheless obeys the von Neumann inequality. That is, if $p$ is an analytic polynomial in three variables, then
\[ \|p(T_1,T_2,T_3)\|\leq \sup\{|p(z_1,z_2,z_3)|: 0\leq |z_1|,|z_2|,|z_3|<1\}. \]
However, there are triples in $\frk{C}_3$ that do not satisfy the von Neumann inequality. A construction of Crabb and Davie \cite[p.~23]{Pisier} provides an example which consists of the three $8\times 8$-matrices,
\[ T_i=\pvect{0         &            &           &           &          &          &          &          \\
              \delta_{i1} &            &           &           &          &          &          &          \\
              \delta_{i2} &            &           &           &          &          &          &          \\
              \delta_{i3} &            &           &           &          &          &          &          \\
                        & -\delta_{i1} & \delta_{i3} & \delta_{i2} &          &          &          &          \\
                        & \delta_{i3} & -\delta_{i2} & \delta_{i1} &          &          &          &          \\
                        & \delta_{i2} & \delta_{i1} & -\delta_{i3} &          &          &          &          \\
                        &            &           &           & \delta_{i1} & \delta_{i2} & \delta_{i3} & 0 } \quad i=1,2,3, \]
where every non-specified entry is 0. These are commuting partial isometries and
\[ T_iT_i^* = \diag(0,\delta_{i1},\delta_{i2},\delta_{i3},1,1,1,1). \]

\begin{proposition}
 The Crabb-Davie triple is extremal.
\end{proposition}
\begin{proof}
 Let $X\in\frk{C}_3$ be an extension of $T$ so
\[ X_i=\pvect{T_i & A_i \\ 0 & B_i},  \quad i=1,2,3,\]
where $A_i\in\mc{L}(\mc{H},\CC^8)$ and $B_i\in\mc{L}(\mc{H})$ for some Hilbert space $\mc{H}$. In order for $X_i$ to be a contraction, we need in particular that
\[ T_iT_i^*+A_iA_i^*\leq 1, \quad i=1,2,3. \]
This implies that $\ran A_i\sbse \ran(1-T_iT_i^*)$. Since
\[ 1-T_iT_i^*=\diag(1,1-\delta_{1i},1-\delta_{2i},1-\delta_{3i},0,0,0,0), \]
we can express the $A_i$ as column vectors whose entries linear functionals on $\mc{H}$;
\[ A_1=\pvect{\eta_1 \\ 0 \\ \phi_1 \\ \psi_1 \\ 0 \\ 0 \\ 0 \\ 0}, \quad A_2=\pvect{\eta_2 \\ \phi_2 \\ 0 \\ \psi_2 \\ 0 \\ 0 \\ 0 \\ 0}, \quad A_3=\pvect{\eta_3 \\ \phi_3 \\ \psi_3 \\ 0 \\ 0 \\ 0 \\ 0 \\ 0}. \]

Notice that
\[ X_iX_j=\pvect{T_iT_j & T_iA_j+A_iB_j \\ 0 & B_iB_j}, \quad i=1,2,3. \]
Therefore $[X_i,X_j]=0$ for all $i,j$ if and only if
\[ [B_i,B_j]=0 \quad T_iA_j+A_iB_j=T_jA_i+A_jB_i. \]
The second series of equations can be expressed as equalities of certain column vectors;
{\scriptsize
\begin{equation}\label{CD:Eq1}
 \pvect{\eta_1\of B_2 \\ \eta_2 \\ \phi_1\of B_2 \\ \psi_1\of B_2 \\ -\phi_2 \\ \psi_2 \\ 0 \\ 0}= \pvect{\eta_2\of B_1 \\ \phi_2\of B_1 \\ \eta_1 \\ \psi_2\of B_1 \\ \psi_1 \\ -\phi_1 \\ 0 \\ 0}, \qquad \pvect{\eta_3\of B_1 \\ \phi_3\of B_1 \\ \psi_3\of B_1 \\ \eta_1 \\ \phi_1 \\ 0 \\ -\psi_1 \\ 0}=\pvect{\eta_1\of B_3 \\ \eta_3 \\ \phi_1\of B_3 \\ \psi_1\of B_3 \\ -\phi_3 \\ 0 \\ \psi_3 \\ 0}, \qquad \pvect{\eta_2\of B_3 \\ \phi_2\of B_3 \\ \eta_3 \\ \psi_2\of B_3 \\ 0 \\ -\psi_3 \\ \phi_3 \\ 0}=\pvect{\eta_3\of B_2 \\ \phi_3\of B_2 \\ \psi_3\of B_2 \\ \eta_2 \\ 0 \\ \phi_2 \\ -\psi_2 \\ 0}
\end{equation}
}By stringing together equations from the 5th through 7th rows of \eqref{CD:Eq1}, we find
\[ \psi_1=-\phi_2=\psi_3=-\psi_1, \]
\[ \psi_2=-\phi_1=\phi_3=-\psi_2.\]
Thus $\psi_i=\phi_i=0$ for all $i$. From the 2nd through 4th rows of \eqref{CD:Eq1}
\[ \eta_1=\phi_1\of B_2=0, \quad \eta_2=\phi_2\of B_1=0, \quad \eta_3=\phi_3\of B_1=0. \]
Thus $A_i=0$ for each $i$, and so $X$ is a trivial extension of $T$.
\end{proof}

\begin{remark}
    We take a moment to show that the Crabb-Davie example does not satisfy the conditions of Lemma \ref{GR:L1}. Note that $\Ran T=\{0\}\oplus \CC^{\oplus 7}$ and $\Ker T=\{0\}^{\oplus 7}\oplus\CC$. Therefore $(\Ran T)^\bot\cap\Ker T=\{0\}$.
\end{remark}

\section{The Varopoulos Example}

We need to establish some notation. Let $J$ be a set, and given $\alpha\in J$ and $x\in\ell^2(J)$, let $x(\alpha)$ denote the $\alpha$-component of $x$. Noting that a linear operator from $\CC$ to $\ell^2(J)$ is uniquely determined by its value at 1, we view the elements of $\ell^2(J)$ as bounded operators $\CC\to\ell^2(J)$ and the linear functionals on $\ell^2(J)$ as bounded operators $\ell^2(J)\to\CC$, the operator adjoint $x\mapsto x^*$ mapping between these. Given $x,y\in\ell^2(J)$ we may now write $xy^*$ for the rank one operator $h\mapsto \inner{h}{y}x$, and $y^*x=\inner{x}{y}$. Another operation we define on $\ell^2(J)$ is the conjugation
\[ \cc{x}(\alpha)=\cc{x(\alpha)}, \quad \alpha \in J. \]
Note that $\cc{x}^*y=\cc{y}^*x$.

Another triple that fails to obey the von Neumann inequality is provided by Varopoulos \cite[p.~86]{Pisier}. Define the Hilbert space $\mc{H}=\CC\oplus \ell^2(J)\oplus \CC$ and let $x_1,x_2,x_3$ be in the unit ball of $\ell^2(J)$. The Varopoulos example consists of the three operators $T_1,T_2,T_3\in\mc{L}(\mc{H})$ defined by
\begin{equation}\label{TVE:DefOfT}
    T_i=\pvect{0 & 0 & 0 \\ x_i & 0 & 0 \\ 0 & \cc{x}_i^* & 0}, \quad i=1,2,3.
\end{equation}
The $T_i$ commute because $\cc{x}_i^*x_j=\cc{x}^*_jx_i$ for $i,j=1,2,3$. The identity
\begin{equation}\label{TVE:Eq1}
    T_iT_i^*=\diag(0,x_ix_i^*,\|x_i\|^2)
\end{equation}
implies $\|T_i\|=\|x_i\|\leq 1$ for each $i$, and so $T\in\frk{C}_3$.

While each $J$ and each triple $x_1,x_2,x_3$ in the unit ball of $\ell^2(J)$ define a $T$ in $\frk{C}_3$, only certain choices of $J$ and $(x_1,x_2,x_3)$ produce an extremal triple. Before providing triples that are in the boundary, we show that we may limit our attention to certain special cases. One restriction we immediately make is to limit ourselves to $\|x_i\|=1$ for each $i$. Indeed, Lemma \ref{GR:L2} and \eqref{TVE:Eq1} imply that $T$ cannot be extremal if $\|x_i\|<1$ for some $i$. Under this restriction $\eqref{TVE:Eq1}$ shows each $T_i$ is a partial isometry.

Another immediate restriction we make is on the size of $\ell^2(J)$. Define the subspace $\mc{R}\sbse\ell^2(J)$ by
\begin{equation}\label{TVE:DefOfR}
    \mc{R}=\Span\{x_1,x_2,x_3,\cc{x}_1,\cc{x}_2,\cc{x}_3\}.
\end{equation}
If $\mc{R}$ is a proper subspace of $\ell^2(J)$ then $T$ cannot be extremal. Indeed,
\[ \ker T_i=\{0\}\oplus \{\cc{x}_i\}^\bot\oplus \CC \qtextq{and} \ran T_i=\{0\}\oplus \CC x_i\oplus \CC \]
for each $i$, and therefore
\[ (\Ran T)^\bot\cap\Ker T= \{0\}\oplus \mc{R}^\bot\oplus \{0\}. \]
Applying Lemma \ref{GR:L1} we find that $T$ cannot be extremal if $\mc{R}^\bot\neq \{0\}$. Therefore we limit our attention to the case $\mc{R}=\ell^2(J)$. With $r=\dim\mc{R}$, we note that $\mc{R}$ is finite dimensional and we fix an orthonormal basis $e_1,\dots,e_r\in\mc{R}$ with the property that $\cc{e_i}=e_i$ for each $i$. 

Any extension $X\in\frk{C}_3$ of $T$ takes the form
\begin{equation}\label{TVE:ExtDef}
    X_i=\pvect{0 & 0 & 0 & \phi_i \\ x_i & 0 & 0 & C_i \\ 0 & \cc{x}_i^* & 0 & 0 \\ 0 & 0 & 0 & B_i}
\end{equation}
acting on $\mc{H}\oplus\mc{M}$ for some Hilbert space $\mc{M}$ where $C_i\in\mc{L}(\mc{M},\mc{R})$, $B_i\in\mc{L}(\mc{M})$, and $\phi_i$ is a linear functional on $\mc{M}$ for $i=1,2,3$. The third entry of the fourth column is 0 because $\|X_i\|\leq 1$ and $\|x_i\|=1$ for each $i$. A second consequence of the inequality $\|X_i\|\leq 1$ is
\[ C_iC_i^*\leq 1-x_ix_i^*, \quad i=1,2,3. \]
This implies $x_i^*C_i=0$ for each $i$. The condition that $X_iX_j=X_jX_i$ is equivalent to requiring
\[ \phi_i\of B_j=\phi_j\of B_i, \quad B_iB_j=B_jB_i \]
\begin{equation}\label{TVE:CC1}
    x_i\phi_j+C_iB_j=x_j\phi_i+C_jB_i,
\end{equation}
\begin{equation}\label{TVE:CC2}
    \cc{x}_i^*C_j=\cc{x}_j^*C_i
\end{equation}
for $i,j=1,2,3$, where $x_i\phi_j$ denotes the map $h\mapsto \phi_j(h)x_i$. Observe that  $x_i^*C_i=0$ implies $C_i^*x_i=0$, and that $\cc{x}_i^*C_j=\cc{x}_j^*C_i$ is equivalent to $C_j^*\cc{x}_i=C_i^*\cc{x}_j$ for all $i$ and $j$.

Define $h^{(i)}_j=C_i^*e_j$ and write $x_i=\sum_{\ell=1}^r a_{i\ell}e_\ell$ for $i=1,2,3$ and $j=1,2,\dots,r$. Then $C_i^*x_i=0$ ($i=1,2,3$) and \eqref{TVE:CC2} become a homogeneous system of linear equations in the vectors $h^{(i)}_\ell$.
\begin{equation}\label{TVE:Systm1}
    \begin{array}{cc}
        a_{i1}h^{(i)}_1+ \dots + a_{ir}h^{(i)}_r =0 & (i=1,2,3) \\
        \cc{a_{i1}}h^{(j)}_1+\dots+\cc{a_{ir}}h^{(j)}_r = \cc{a_{j1}}h^{(i)}_1+\dots+\cc{a_{jr}}h^{(i)}_r & (i,j=1,2,3)
    \end{array}
\end{equation}
Let $\Lambda$ denote the $6\times 3r$ scalar matrix representing this linear system;
\begin{equation*}
        \Lambda=\begin{pmatrix}
        a_{11} & \cdots & a_{1r} &        &       &        &        &       &        \\
               &        &        & a_{21} & \dots & a_{2r} &        &       &        \\
               &        &        &        &       &        & a_{31} & \dots & a_{3r} \\
        \cc{a_{21}} & \dots & \cc{a_{2r}} & -\cc{a_{11}}   & \dots & -\cc{a_{1r}} &              &       &              \\
        \cc{a_{31}} & \dots & \cc{a_{3r}} &                &       &              & -\cc{a_{11}} & \dots & -\cc{a_{1r}} \\
                    &       &             & \cc{a_{31}}    & \dots & \cc{a_{3r}}  & -\cc{a_{21}} & \dots & -\cc{a_{2r}}
       \end{pmatrix}
    \end{equation*}
    where every non-specified entry is 0.

\begin{lemma}\label{TVE:Lem}
    If $\Lambda$ has a non-trivial kernel, then $T$ is not extremal.
\end{lemma}
\begin{proof}
    Assume that $\Lambda$ has non-trivial kernel and set $\mc{M}=\CC$, $\phi_i=0$, $B_i=0$ for $i=1,2,3$. Then the commutivity of $X$ is determined entirely by \eqref{TVE:CC2} and contractivity entirely by $C_iC_i^*+x_ix_i^*\leq 1$ for each $i$. As $\Lambda$ has a non-trivial kernel and $\mc{M}=\CC$, there is a non-zero solution to \eqref{TVE:Systm1}. Thus $C_i=\sum_{j=1}^r e_jh^{(i)*}_j$ is non-zero. Since the kernel of $\Lambda$ is linear, we may assume $\|C_i\|\leq 1$.
\end{proof}

\begin{proposition}
    The Varopoulos triple $T$ is extremal if and only if $\dim\mc{R}=2$ and $\ker\Lambda=\{0\}$.
\end{proposition}
\begin{proof}
    Suppose first that $T$ is extremal. Then Lemma \ref{TVE:Lem} implies $\ker\Lambda= \{0\}$. Since $r>2$, or rather $3r>6$, implies that $\Lambda$ has a non-trivial kernel, it follows that $r\leq 2$. In the case that $r=1$, there are $c_i\in\CC$ with $|c_i|=1$ so that $x_i=c_ix_1$ for $i=1,2,3$. Setting $\mc{M}=\CC$ and $B_i=0,C_i=0$, $\phi_i=c_i$ in \eqref{TVE:ExtDef} for each $i$ yields a non-trivial extension.
    
    Conversely, suppose $r=2$ and $\ker\Lambda=\{0\}$, in which case
    \begin{equation}\label{TVE:Lambdar2}
        \Lambda=\begin{pmatrix}
        a_{11} & a_{12} &  &  &  &  \\
                    &             & a_{21} & a_{22} &  &  \\
                    &             &       &       & a_{31} & a_{32} \\
        \cc{a_{21}} & \cc{a_{22}} & -\cc{a_{11}} & -\cc{a_{12}} &  &  \\
        \cc{a_{31}} & \cc{a_{32}} &                   &                   & -\cc{a_{11}} & -\cc{a_{12}} \\
                         &                  & \cc{a_{31}}  & \cc{a_{32}} & -\cc{a_{21}} & -\cc{a_{22}}
       \end{pmatrix}
    \end{equation}
    where every non-specified entry is 0. We have $\det\Lambda\neq 0$ and this easily implies $\dim\Span\{x_1,x_2,x_3\}>1$. Suppose $X\in\frk{C}_3$ is an extension of $T$, written as in \eqref{TVE:ExtDef}. Since $\det\Lambda\neq 0$ it follows that \eqref{TVE:Systm1} has only the trivial solution, so that $C_i^*e_j=h_j^{(i)}=0$ for all $i,j$ and therefore $C_1=C_2=C_3=0$. Then \eqref{TVE:CC1} yields $x_i\phi_j=x_j\phi_i$ for each $i,j$. The set $\{x_1,x_2,x_3\}$ does not generate a space of dimension 1, and therefore $\phi_i=0$ for $i=1,2,3$. We conclude that $X$ is trivial and thus $T$ is extremal.
\end{proof}

\begin{remark}
    We demonstrate that the condition $\ker\Lambda=\{0\}$ is not automatically satisfied when $r=2$. Using the matrix representation of $\Lambda$ given in \eqref{TVE:Lambdar2}, one finds that $\det\Lambda=1$ for the vectors $(1,0)$, $(1/\sqrt{2},1/\sqrt{2})$, $(0,1)$. On the other hand, when the vectors $x_1,x_2,x_3$ are instead $(1,0)$, $(1,0)$, $(0,1)$, one finds $\det\Lambda=0$. 
\end{remark}

\end{document}